\documentclass[a4paper]{amsart}
\usepackage{amssymb,bm}
\usepackage{amsmath}
\usepackage{amsfonts}
\usepackage{amsthm}
\usepackage{color}
\usepackage{graphics,graphicx,url}
\usepackage{enumerate}
\usepackage{comment}
\usepackage{centernot}
\usepackage[normalem]{ulem}

\usepackage{breqn} 

\usepackage[a4paper]{geometry}
\usepackage{tikz}
\usetikzlibrary{calc,decorations.pathmorphing,patterns,arrows,decorations.markings,positioning}
\usetikzlibrary{automata,chains,fit,shapes}

\usepackage{graphicx,pdflscape,rotating,float}
\usepackage{enumitem}

\newtheorem{theorem}{Theorem}
\newtheorem{lemma}[theorem]{Lemma}
\newtheorem{conjecture}[theorem]{Conjecture}

\newtheorem{corollary}[theorem]{Corollary}
\theoremstyle{definition}
\newtheorem{remark}[theorem]{Remark}
\newtheorem{definition}[theorem]{Definition}
\newtheorem{example}[theorem]{Example}

\newcommand{\N}{\mathbb N}
\newcommand{\Z}{\mathbb Z}
\newcommand{\iec}{IEC}

\usetikzlibrary{positioning,calc}

\usepackage[numbers]{natbib}

\usepackage{freetikz}

\begin{document}

\title{Rewriting systems, plain groups, and geodetic graphs}
\thanks{Research supported by  Australian Research Council grant  DP210100271}

 \author[M. Elder]{Murray Elder}
 \address{University of Technology Sydney, Ultimo NSW 2007, Australia}  
 \email{murray.elder@uts.edu.au}

 \author[A. Piggott]{Adam Piggott}
 \address{Australian National University, Canberra ACT  0200, Australia}
\email {adam.piggott@anu.edu.au}

\date{\today}



\maketitle             
\begin{abstract}
We prove that a group is presented by finite convergent length-reducing rewriting systems where each rule has left-hand side of length 3 if and only if the group is {plain}. 
Our proof goes via a new result concerning properties of embedded circuits in geodetic graphs, which may be of independent interest in graph theory.

\end{abstract}

 \section{Introduction}

 The study of rewriting systems connects abstract algebra and theoretical computer science in deep and useful ways. A program of research initiated in the 1980s seeks to characterise algebraically the families of groups that may be presented by various families of rewriting systems (see \cite{MadlenerOttoDescriptivePower} for a broad introduction). An important part of this program is to characterise the groups that may be presented by length-reducing rewriting systems.  Early progress was swift.
 Diekert \cite{DiekertLengthReducing} (see also \cite{MadlenerOttoCommutativity}) proved that that the family of groups admitting presentation by finite convergent length-reducing rewriting systems is properly contained within the family of virtually-free groups;  Avenhaus, Madlener and Otto \cite{AvenhausMadlenerOtto} proved that the family of groups admitting presentation by finite convergent length-reducing rewriting systems in which each rule has a left-hand-side of length two is exactly the family of plain groups (a group is \emph{plain} if it isomorphic to a free product of finitely-many factors, with each factor a finite group or an infinite cyclic group); an explicit construction (described in Section \ref{sec:Rewriting}) shows that any plain group admits presentation by a finite convergent length-reducing rewriting system.
From such results the plain groups emerged as the likely family of groups presented by finite convergent length-reducing rewriting systems. In 1987, Madlener and Otto \cite{MadlenerOttoLengthReducing} summarised the state of knowledge by highlighting the following  two conjectures, the resolution of which  would ``give a complete algebraic characterisation of groups presented by length-reducing systems''.

\begin{conjecture}[Gilman \cite{Gilman}]\label{Conj:Gilman}
Let $G$ be a group.  Then $G$  admits presentation by a finite convergent length-reducing rewriting system $(\Sigma, T)$  in which the right-hand side of every rule has length  at most one
 if and only if  $G$ is plain.
\end{conjecture}

\begin{conjecture}[Madlener and Otto \cite{MadlenerOttoLengthReducing}]\label{Conj:LR}
Let $G$ be a group.  Then $G$  admits presentation by a finite convergent length-reducing rewriting system $(\Sigma, T)$  if and only if  $G$ is plain.
\end{conjecture}

Although a special case of Conjecture \ref{Conj:LR}, Gilman's Conjecture was important enough to consider separately because it seemed more tractable and its resolution may provide clues to the more general problem. The recent positive solution to Gilman's Conjecture by Eisenberg and the second author  \cite{GilmansConjecture}  motivates the present work. Our main result proves Conjecture \ref{Conj:LR} in a special case not implied by \cite{GilmansConjecture}.

\begin{theorem}\label{thmA:main}
Let $G$ be a group. Then $G$  admits presentation by a finite convergent length-reducing rewriting system $(\Sigma, T)$ such that 
$\Sigma=\Sigma^{-1}$ and 
 the left-hand side of every rule has length at most three 
 if and only if  $G$ is plain.
\end{theorem}

Our proof  is essentially graph theoretic, and exploits the fact that if $G$ and $(\Sigma, T)$ are as in the theorem, then the undirected Cayley graph $\Gamma = \Gamma(G, \Sigma)$ is geodetic.  A simple undirected graph $\Gamma$ is \emph{geodetic} if between any pair of vertices there exists a unique shortest path.  In \cite[Problem 3, p.105]{Ore}, Ore posed the problem of giving a general classification of all finite geodetic graphs, but that has proven very difficult.  Although planar geodetic graphs have been characterised \cite{StempleWatkins}, various structural aspects of geodetic graphs of diameter two and three are understood \cite{Stemple,GeodeticBlocksDiameterThree, Scapellato}, the geodetic graphs homeomorphic to complete graphs are known \cite{StempleHomeoToKn}, and a number of clever procedures have been developed for constructing new geodetic graphs from existing ones (see, for example, \cite{Frasser2016GeodeticGH}), a general classification of geodetic graphs is not close.  We prove the following, which is new and may be of independent interest simply because the task of classifying geodetic graphs has proven to be  so difficult.

\begin{theorem}
\label{ThmB:Shortcircuits}
If $\Gamma$ is 
an undirected simple geodetic graph in which isometrically embedded circuits have length at most five, 
 then 
all embedded circuits have diameter at most two. 
\end{theorem}

While Theorem~\ref{thmA:main} falls well short of resolving Conjecture~\ref{Conj:LR}, and Theorem~\ref{ThmB:Shortcircuits} is an incremental  contribution to our understanding of geodetic graphs, we think our proof offers insight into the difficulties to be overcome by any argument that takes a primarily graph-theoretic approach to a significant open problem that has  defied the  efforts of many authors for more than three decades.

\section{Definitions}\label{sec:Defns}

\subsection{Rewriting systems}\label{sec:Rewriting}

A {\em rewriting system} is a pair $(\Sigma, T)$ that formalises the idea of working with products from a set of allowable symbols, using a set of simplifying rules.  The set $\Sigma$ is a nonempty set, called an {\em alphabet}; its elements are called {\em letters}.  We write $\Sigma^\ast$ for the set of all finite words, including the empty word  $\lambda$, that can be made using letters from the alphabet.  For any $w \in \Sigma^\ast$, we write $|w|$ for the {\em length} of $w$; $\lambda$ is the unique word of length 0.  The second element $T$ is a possibly empty subset of $\Sigma^*\times \Sigma^*$, called a set of {\em rewriting rules}. The set of rewriting rules determines a relation $\to$ (read ``immediately reduces to'') on the set $\Sigma^\ast$ by the following rule: $a \to b$ if $a=u\ell v$,  $b=urv$ and $(\ell, r) \in T$.  The reflexive and transitive closure of $\to$ is denoted $\overset{\ast}{\to}$ (read ``reduces to''). 
Thus  the rewriting rules specify allowable  factor replacements, and $u \overset{\ast}{\to} v$ if $v$ can be obtained from $u$ by a sequence of allowable factor replacements.  A word $u \in \Sigma^\ast$ is {\em irreducible} if no factor of $u$ is the left-hand side of any  rewriting rule, and hence $u \overset{\ast}{\to} v$ implies that $u = v$.

The reflexive, transitive and symmetric closure of $\to$ is called ``equivalence'', and denoted $\overset{\ast}{\leftrightarrow}$.  The operation of concatenation of representatives is well defined on the set of $\overset{\ast}{\leftrightarrow}$-equivalence classes, and hence defines a quotient monoid $M = M(\Sigma, T)$.  We say that $M$ is the monoid presented by $(\Sigma, T)$.  When the equivalence class of every letter (and hence also the equivalence class of every word) has an inverse, the monoid $M$ is a group and we say it is {\em the group presented by $(\Sigma, T)$}.

\begin{example}\label{eg:InfiniteCyclicGroup}
Let $\Sigma = \{a, A\}$ and let
$T = \left\{(a A, \lambda), (Aa, \lambda)\right\}.$
Then $(\Sigma, T)$ presents a group isomorphic to $\Z$, the infinite cyclic group.

\end{example}

\begin{example}\label{eg:FiniteGroup}
Let $G$ be a finite group, let $\Sigma = G \setminus \{e_G\}$ and let 
\[T =\left\{(gh, k) \mid g, h, k \in \Sigma \text{ and } gh =_G k\} \cup \{(gh, \lambda) \mid g,h \in \Sigma \text{ and } g =_G h^{-1}\right\}.\]
Then $(\Sigma, T)$ presents a group isomorphic to $G$.
\end{example}

A rewriting system $(\Sigma, T)$ is {\em finite} if $\Sigma$ and $T$ are finite sets,  {\em terminating} (or {\em noetherian)} if there are no infinite sequences of allowable factor replacements, and {\em length-reducing} if for all $(\ell, r) \in T$ we have that $|\ell| > |r|$.  It is clear that length-reducing rewriting systems are terminating.
A rewriting system is called {\em confluent} if for all $w,x,y\in \Sigma^*$, if  $w\overset{\ast}{\to} x$ and $w\overset{\ast}{\to} y$ then there exists $z\in \Sigma^*$ such that $x\overset{\ast}{\to} z$ and $y\overset{\ast}{\to} z$. 
A rewriting system is called {\em convergent} if it is terminating and confluent.   The following lemma (see, for example, \cite[Theorem 1.13, p.13]{BookOtto}) illustrates the utility of convergent rewriting systems.

\begin{lemma}\label{lem:ConvergentRewriting}
In a convergent rewriting system, rewriting any word in $\Sigma^*$ until you can rewrite no more is an algorithm for producing the unique irreducible word (the normal form) representing the same element.
\end{lemma}

The following simple lemma is provided without proof.  The corollary is easily proved by applying the lemma to the rewriting systems exhibited in Examples~\ref{eg:InfiniteCyclicGroup} and \ref{eg:FiniteGroup}.

\begin{lemma}[Combining rewriting system to present free products]\label{eg:FreeProduct}
Suppose that $(\Sigma_1, T_1), \dots$, $(\Sigma_n, T_n)$ are rewriting systems presenting groups $G_1, \dots, G_n$ respectively and such that the alphabets $\Sigma_1, \dots, \Sigma_n$ are pairwise disjoint.  The combined rewriting system
$\left( \cup_{i=1}^n \Sigma_i, \cup_{i=1}^n T_i\right)$
presents the free product $G_1 \ast \cdots \ast G_n$. 
\end{lemma}

\begin{corollary}\label{cor:Plain}
If $G$ is a plain group, then $G$ admits presentation by a finite convergent length-reducing rewriting system $(\Sigma, T)$ where 
$\Sigma=\Sigma^{-1}$ and the left-hand side of every rule has length equal to two.
\end{corollary}

 \subsection{Graph theory}\label{sec:GraphTheory}
 
A simple undirected graph $\Delta$ is a pair comprising a nonempty set $V(\Delta)$, the set of {\em vertices}, and a set of two-element subsets $E(\Delta)$, the set of {\em edges}.  The vertices that form an edge are said to be {\em adjacent}. All graphs considered in this paper will be simple and undirected. For the remainder of this section, fix a simple undirected graph $\Delta$.

A {\em path} of length $n$  in $\Delta$ from  a vertex $u$ to  a vertex $v$ is a sequence of vertices $u=u_0, u_1, \dots, u_n=v$ with the property that $u_{i-1}$ and $u_i$ are adjacent for $i = 1, \dots, n$.  A path from $u$ and $v$ is called a \emph{geodesic} if there is no shorter path in $\Delta$ from $u$ to $v$.  If for each pair $(u, v)$ of distinct vertices in $\Delta$ there is at least one path in $\Delta$ from $u$ to $v$, we say that $\Delta$ is {\em connected}; if for each pair $(u, v)$ of distinct vertices in $\Delta$ there exists a unique geodesic from $u$ to $v$, we say that $\Delta$ is \emph{geodetic}.  If $\Delta$ is connected, there is a natural metric $d$ on the vertex set of $\Delta$ such that $d(u, v)$ is the length of a shortest path in $\Delta$  from $u$ to $v$.

A {\em circuit} is a path $u_0, u_1, \dots, u_n$ where $u_0=u_n$. A {\em sub-path} of a circuit  $u_0, u_1, \dots, u_n$ is either a path $u_i,\dots u_j$ where $0\leq i\leq j\leq n$ or a path $u_i,\dots, u_n, u_1,\dots, u_j$ where $1\leq j\leq i\leq n$. 
 A circuit $u_0, u_1, \dots, u_n$ is  {\em embedded} if the vertices $u_0, \dots, u_{n-1}$ are distinct.
 An embedded circuit  in $\Delta$ is {\em isometrically embedded} if the subgraph comprising the vertices in the circuit and the edges between consecutive vertices is convex in $\Delta$; that is, $d(u_i, u_j) = \min\{j-i, n+i-j\}$ for all $0 \leq i < j < n$. 
 We will use the acronym \iec\ for {\em isometrically embedded circuit}. We note that if $u, v$ are adjacent vertices in $\Delta$, then the path $u, v, u$ is an isometrically embedded circuit of length two.  We also note that in a geodetic graph, the unique geodesic joining two vertices of  an \iec\ is a subpath of the \iec.

A vertex $v$ in $\Delta$ is a {\em cut vertex} if $\Delta$ is connected, but the graph obtained from $\Delta$ by removing $v$ and the edges incident to $v$ is disconnected.   A graph is two-connected if it is connected and has no cut vertices.   The maximal two-connected subgraphs of a graph $\Delta$ are called {\em blocks}.  It follows immediately from the maximality of blocks that any block $B$ in $\Delta$ is the subgraph of $\Delta$ induced by the vertex set of $B$.  In a connected graph having at least two vertices,  each block has at least two vertices.  The following well-known  characterisation of blocks (see, for example, \cite[Theorem 5.4.3, p. 87]{Ore}) will be useful throughout our argument.

\begin{lemma}\label{lem:Whitman}
 Let $\Delta$ be a simple undirected graph. Two vertices $u, v$ of $\Delta$  lie in the same block if and only if there exists an embedded circuit in $\Delta$ that visits both.
 \end{lemma} 

Given a connected graph $\Delta$, the {\em block-cut tree} $T= T(\Delta)$ is a well-known construction which encodes the block structure of $\Delta$.  The graph $T$ has one vertex $v_x$ (of type I) for each vertex $x$ of $\Delta$, and one vertex $v_B$ (of type II) for each block $B$ of $\Delta$; a type I vertex $v_x$ is adjacent in $T$ to a type II vertex $v_B$ if $x$ is a vertex in the block $B$.  For any connected graph $\Delta$, the block-cut tree $T(\Delta)$ is a tree (a connected graph in which every embedded circuit has length at most two).  See for example Figure~\ref{fig:BCT}.

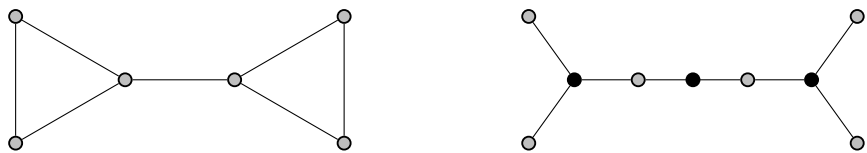
\begin{figure}[!ht]
	\centering

\smallskip
 \begin{tikzpicture}[rotate=90,
  scale=0.6]

   \node[dot] (d01) at (0,1.4) {};
   \node[dot] (d02) at (1.4,-1 ) {};
   \node[dot] (d03) at (-1.4,-1) {};
 
      \draw(d02) to (d01); 
         \draw (d03) to (d02); 
             \draw (d01) to (d03);

   \node[dot] (d11) at (0,3.8) {};
   \node[dot] (d12) at (1.4,6.2) {};
   \node[dot] (d13) at (-1.4,6.2) {};
 
     \draw (d12) to (d11); 
         \draw (d13) to (d12); 
             \draw (d11) to (d13); 
                 \draw (d01) to (d11);          
 \end{tikzpicture}
 \hspace{20mm}
  \begin{tikzpicture}[rotate=90,
   scale=0.6]

     \node[dot, fill=black!100] (e) at (0,2.6) {};
        \node[dot, fill=black!100] (d10) at (0,5.2) {};
           \node[dot, fill=black!100] (d00) at (0,0) {};
           
   \node[dot] (d01) at (0,1.4) {};
   \node[dot] (d02) at (1.4,-1 ) {};
   \node[dot] (d03) at (-1.4,-1) {};

     \draw (d00) to (d01); 
         \draw (d00) to (d02); 
             \draw (d00) to (d03);    

   \node[dot] (d11) at (0,3.8) {};
   \node[dot] (d12) at (1.4,6.2) {};
   \node[dot] (d13) at (-1.4,6.2) {};

                 \draw (d10) to (d11); 
         \draw (d10) to (d12); 
             \draw (d10) to (d13); 
                 \draw (d01) to (d11);

 \end{tikzpicture}
 
	\caption{Example of a graph and its block-cut tree. Type II vertices are solid black.}
	\label{fig:BCT}
\end{figure}

 \subsection{Key lemma and broomlike graphs}
 
The following lemma and its proof 
are  paraphrased from \cite[Proposition 6.3]{EP}.

\begin{lemma}
\label{lem:EP}
 Let $\Gamma$ be a geodetic graph, and let $u_0, u_1, \dots, u_n$ and $u_0, u'_1, \dots, u'_n$
 be equal length geodesics in $\Gamma$ such that $u_1 \neq u'_1$ and $d(u_n, u'_n) = 1$.  Then 
 \[u_0, u_1, \dots, u_n, u'_n, \dots, u'_1, u_0\] is an \iec. \end{lemma}

 \begin{figure}[!ht]
	\centering

	\begin{tikzpicture}[auto,node distance=.5cm,
    latent/.style={circle,draw,very thick,inner sep=0pt,minimum size=30mm,align=center},
    manifest/.style={rectangle,draw,very thick,inner sep=0pt,minimum width=45mm,minimum height=10mm},
    paths/.style={->, ultra thick, >=stealth'},
]

    \node[dot] (0)  [label={[xshift=-.7cm, yshift=-.3cm]$u_0=v_0$}] {};
    
    \node[dot] (1) [label={[xshift=0cm, yshift=0cm]$v_1$} ] at ($ (0)   + (40:1)$) {};

        \node[dot] (2)  [label={[xshift=0cm, yshift=0cm]$v_2$}, above right = .2em and 3em of 1]  {};
                \node[dot] (3)  [label={[xshift=0cm, yshift=0cm]$v_{n-1}$}, below  right = .1em and 10em of 2]  {};
      
                                \node[dot] (4)  [label={[xshift=0.6cm, yshift=0cm]$v_n=u_n$}, below right = .3em and 2.2em of 3]  {};
    \node[dot] (11)  [label={[xshift=0cm, yshift=-.65cm]$v_{2n}$}, below right = 2em and 1.5em of 0] {};
        \node[dot] (12)  [label={[xshift=0.1cm, yshift=-.67cm]$v_{2n-1}$}, below right = .5em and 2.5em of 11] {};
  \node[dot] (14)   [label={[xshift=0.1cm, yshift=-.670cm]$v_{n+2}$}, below = 4.4em  of 3] {};
        \node[dot] (15) [label={[xshift=0.6cm, yshift=-.670cm]$v_{n+1}=u_n'$},  above right = 1em and 2.3em   of 14]  {};
  
    \path[draw,thick]
    (0) edge node {} (1)
     (1) edge node {} (2)
               (3) edge node {} (4)
    (0) edge node {} (11)
      (11) edge node {} (12)
       (14) edge node {} (15)
                 (4) edge node {}  (15);

            \draw [thick,dashed] (2) to [bend left=15]  (3)
            (12) to [bend right=15]  (14);
  \end{tikzpicture}

	\caption{Geodesics in Lemma~\ref{lem:EP}, relabeled as in the proof.}
	\label{fig:EP}
\end{figure}
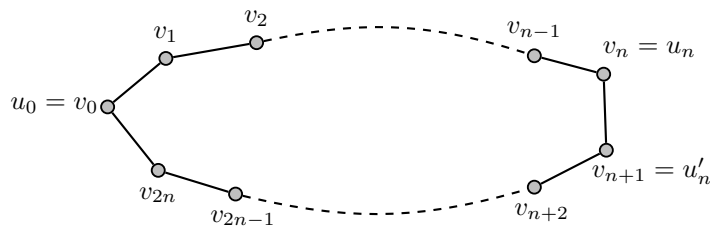

 \begin{proof}
Since $\Gamma$ is geodetic and $u_1 \neq u'_1$, the sets $\{u_1, \dots, u_n\}$ and $\{u'_1, \dots, u'_n\}$ are disjoint.  It is convenient to relabel the vertices $v_0, \dots, v_{2n}$ so that
 \[v_0 = u_0, \dots v_n = u_n, v_{n+1} = u'_n, \dots, v_{2n} = u'_1.\] In what follows we shall consider the index $i$ of a vertex $v_i$  modulo $2n+1$.
 
Using induction, we shall prove the following statement $S(i)$ for all $i$:
The paths
 \[v_i, v_{i+1}, \dots, v_{i+n} \text { and } v_i, v_{i-1}, \dots, v_{i-n}\]
 are geodesics.  The result follows immediately.

That $S(0)$ holds is immediate from the hypotheses.  Suppose that $S(i)$ holds for some index $i$.  It follows that $v_{i+1}, \dots, v_{i+n}$ is the unique geodesic from $v_{i+1}$ to $v_{i+n}$, because it is a subpath of the geodesic $v_i, v_{i+1}, \dots, v_{i+n}$.  It follows immediately that $d(v_{i+1}, v_{i+n}) = n-1$.

If $d(v_{i+1}, v_{i+n+1}) < n$, then there is a path of length at most $n$ from $v_{i}$ to $v_{i+n+1}= v_{i-n}$ through $v_{i+1}$. This contradicts the fact that $v_{i}, v_{i-1}, \dots, v_{i-n}$ is the unique geodesic from $v_{i}$ to $v_{i-n}$. It follows that $d(v_{i+1}, v_{i+n+1}) \geq n$, from which it follows that  $v_{i+1}, v_{i+2}, \dots, v_{i+n+1}$ is the unique geodesic from $v_{i+1}$ to $v_{i+n+1}$. 

If $d(v_{i+1}, v_{i+1-n}) < n$, then there is a path of length at most $n$ from $v_{i+1}$ to $v_{i-n}= v_{i+n+1}$ through $v_{i+1-n}$.  This contradicts the fact, just shown, that $v_{i+1}, v_{i+2}, \dots, v_{i+n+1}$ is the unique geodesic from $v_{i+1}$ to $v_{i+n+1}$.  It follows that $d(v_{i+1}, v_{i+1-n}) \geq n$, from which it follows that and $v_{i+1}, v_{i+1-1}, \dots, v_{i+1-n}$ is the unique geodesic from $v_{i+1}$ to $v_{i+1-n}$.
 \end{proof}

 We make the following definition. Our vocabulary borrows from \cite{Broomlike}.   
 
 \begin{definition}[$s$-broomlike]\label{def:strongfftp}
Let $\Delta$ be a geodetic graph and $s$ a positive integer.
We say that $\Delta$ is {\em $s$-broomlike}
if whenever $a_0, \dots, a_{n-1}, a_n, b$ is a path comprising distinct vertices such that $a_0, \dots, a_n$ is a geodesic but $a_0, \dots, a_n, b$ is not, then the geodesic from $a_0$ to $b$ is $a_0, \dots, a_{n-p}, b_{n-p+1}, \dots, b_n=b$ for $p \leq s$ and $b_{n-p+1} \neq a_{n-p+1}$.
\end{definition}

  \begin{figure}[!ht]
	\centering

\begin{tikzpicture}[auto,node distance=.5cm,
    latent/.style={circle,draw,very thick,inner sep=0pt,minimum size=30mm,align=center},
    manifest/.style={rectangle,draw,very thick,inner sep=0pt,minimum width=45mm,minimum height=10mm},
    paths/.style={->, ultra thick, >=stealth'},
]

   \node[dot] [label={[xshift=-.3cm, yshift=-.3cm]$a_0$}](0) {};
    \node[dot] (1) [label={[xshift=0cm, yshift=0cm]$a_1$}, above right= 2em and 2em of 0]  {};
        \node[dot] (2) [label={[xshift=0cm, yshift=0cm]$a_{n-p}$},above right = 3em and 8em of 1]  {};
                \node[dot] (3) [label={[xshift=0cm, yshift=-.65cm]$b_{n-p+1}$}, below right = 2em and 3.5em of 2]  {};
                       
                \node[dot] (5) [label={[xshift=0cm, yshift=0cm]$a_{n-p+1}$},right = 3em of 2]  {};

 \node[dot] (6x) [label={[xshift=0cm, yshift=0cm]$a_{n-1}$}, right = 4.5em of 5]  {};              
 \node[dot] (4x) [label={[xshift=0cm, yshift=-.7cm]$b_{n-1}$},right = 4.5em of 3]  {};

 \node[dot] (6) [label={[xshift=0cm, yshift=0cm]$a_n$}, right = 2.5em of 6x]  {};              
 \node[dot] (4) [label={[xshift=0cm, yshift=-.7cm]$b$},right = 2.5em of 4x]  {};

    \path[draw,thick]
    (0) edge node {} (1)
    
                (6) edge node {} (6x)
                         (2) edge node {} (5)
     (6) edge node {} (4);

        \path[draw]
              (2) edge node {} (3)
         (4) edge node {} (4x);
    
      \draw [dashed] 
     (3) edge node {} (4x);
           
             \draw [thick,dashed] 
              (5) edge node {} (6x)
              
              (1) to [bend left=20]  (2);
      
\end{tikzpicture}
	\caption{Illustrating the $s$-broomlike property (Definition~\ref{def:strongfftp}).}
	\label{fig:DEFNbroomlike}
\end{figure}
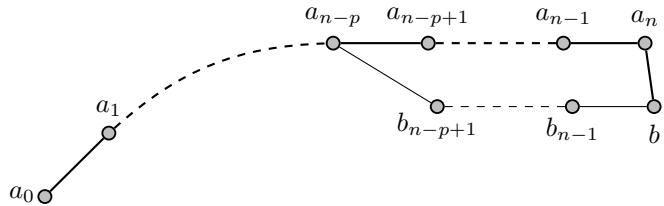

  \begin{lemma}
  \label{lem:IECbroomlike}
Let $\Delta$ be a geodetic graph and 
 $s$ a positive integer.
If every  \iec\ in $\Delta$  has length  at most $2s+1$, 
    then 
$\Delta$ is  $s$-broomlike.
  \end{lemma}
 
 \begin{proof}
Let $a_0, \dots, a_{n-1}, a_n, b$ be a path comprising distinct vertices such that $\alpha=a_0, \dots, a_n$ is a geodesic but $a_0, \dots, a_n, b$ is not.
Let $\beta$ be the geodesic from $a_0$ to $b$, and let $\tau=a_0,\dots, a_{n-p}$ be the longest prefix shared by $\alpha$ and $\beta$, where $0< p\leq n$.
Then 
$\alpha=\tau\alpha'$ and $\beta=\tau\beta'$ with $\alpha'=a_{n-p},a_{n-p+1}\dots, a_n$ and $\beta'=a_{n-p}, b_{n-p+1}, \dots, b$
both geodesics, and  $a_{n-p+1}\neq b_{n-p+1}$ for if not we could have made $\tau$ longer.

Since $\Delta$ is geodetic,  $|\alpha'|=|\beta'|$, so  $b_n=b$.
Then $\alpha',\beta'$ satisfy the hypothesis of Lemma~\ref{lem:EP}, which means
\[a_{n-p},a_{n-p+1}\dots, a_n, b=b_n, b_{n-1},\dots, 
 b_{n-p+1}, a_{n-p}\] is an \iec, so 
 its length is bounded by $2s+1$, which means $|\beta'|=|\alpha'|=p\leq s$.
 \end{proof}

\subsection{Cayley graphs}

An important and much studied connection between graph theory and group theory is via the Cayley graph.  In this article, we consider the undirected Cayley graph corresponding to a group and a choice of finite generating set. For any group $G$ let $e_G$ denote the identity element.

For a group $G$ and a generating set $\Sigma$, the \emph{undirected Cayley graph of $G$ with respect to $\Sigma$} is the  simple undirected graph $\Gamma = \Gamma(G, \Sigma)$ with vertex set $G$ and in which distinct vertices $g, h \in G$ are adjacent if and only if $g^{-1} h \in \Sigma \cup \Sigma^{-1}$.  See for example Figure~\ref{fig:example23}. If $\Sigma$ is finite then $\Gamma$ is locally finite.
Each path $u_0, u_1, \dots, u_n$ in $\Gamma$ is labeled by a word $a_1 \dots a_n \in (\Sigma \cup \Sigma^{-1})^\ast$ where $a_i =_G u_{i-1}^{-1} u_i$.  A geodesic path in $\Gamma$ from $e_G$ to $g$ is a shortest word in $(\Sigma \cup \Sigma^{-1})^\ast$ spelling the group element $g$.

Note that by definition if $x\in\Sigma$ and $x=_Ge_G$ then $x$ will not appear as the label of any edge in $\Gamma(G,\Sigma)$. Also if $x,y\in\Sigma$ and $x=_Gy$ then the unique edge joining adjacent vertices  $g$ to $gx$ in $\Gamma(G,\Sigma)$ may be labeled by either $x$ or $y$. 

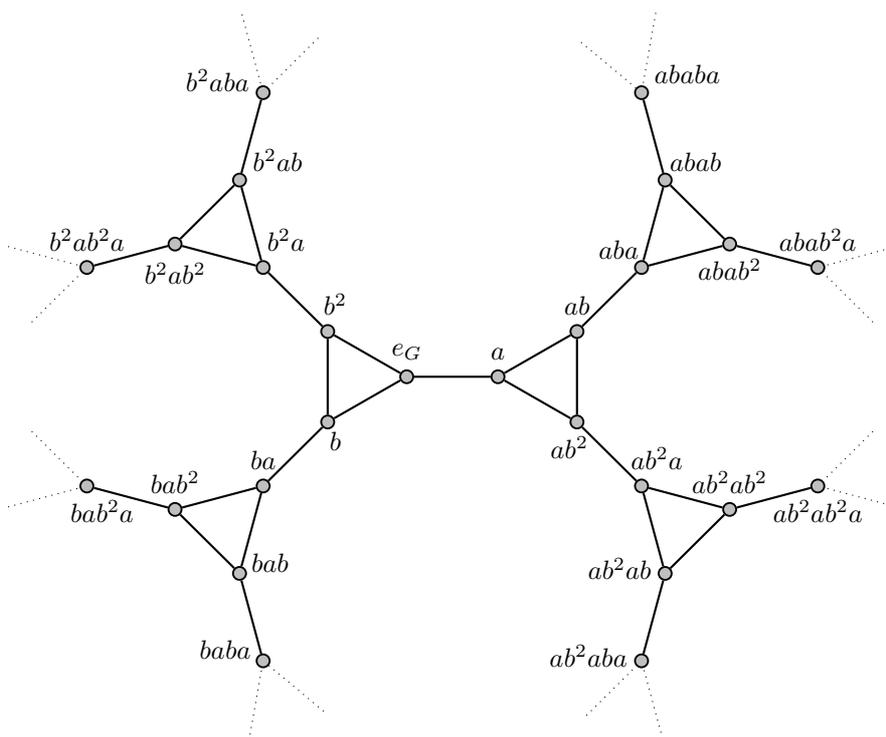
\begin{figure}[!ht]
	\centering

 \begin{tikzpicture}[ scale=0.6]
		 
\node[dot] (d01)  [label={[xshift=0cm, yshift=0cm]$a$}]  {};
\node[dot] (d02)  [label={[xshift=-0cm, yshift=0cm]$ab$}] at ($ (d01)   + (30:2)$) {};
\node[dot] (d03)  [label={[xshift=-.1cm, yshift=-.7cm]$ab^2$}] at ($ (d01)   + (-30:2)$) {};

\node[dot] (d04)  [label={[xshift=-0cm, yshift=0cm]$e_G$}] at ($ (d01)   - (0:2)$) {};
\node[dot] (d05)  [label={[xshift=.1cm, yshift=-.6cm]$b$}] at ($ (d04)   - (30:2)$) {};
\node[dot] (d06)  [label={[xshift=.1cm, yshift=0cm]$b^2$}] at ($ (d04)   - (-30:2)$) {};

\node[dot] (d07)  [label={[xshift=-.3cm, yshift=-.1cm]$aba$}] at ($ (d02)  + (45:2)$) {};
\node[dot] (d08)  [label={[xshift=.4cm, yshift=-.1cm]$abab$}] at ($ (d07)   + (75:2)$) {};
\node[dot] (d09)  [label={[xshift=0cm, yshift=-.7cm]$abab^2$}] at ($ (d07)   +(15:2)$) {};

\node[dot] (d10)  [label={[xshift=.2cm, yshift=0cm]$ab^2a$}] at ($ (d03)  + (-45:2)$) {};
\node[dot] (d11)  [label={[xshift=-.6cm, yshift=-.3cm]$ab^2ab$}] at ($ (d10)   + (-75:2)$) {};
\node[dot] (d12)  [label={[xshift=0cm, yshift=0cm]$ab^2ab^2$}] at ($ (d10)   +(-15:2)$) {};

   \draw [thick](d02) to (d01); 
            \draw [thick] (d03) to (d02); 
                \draw [thick] (d01) to (d03); 
         \draw [thick](d01) to (d04); 
            \draw [thick] (d04) to (d05); 
                \draw [thick] (d05) to (d06); 
         \draw [thick](d06) to (d04); 
         \draw [thick](d02) to (d07); 
            \draw [thick](d07) to (d08); 
                  \draw [thick](d07) to (d09); 
                        \draw [thick](d08) to (d09); 
                              \draw [thick](d03) to (d10); 
            \draw [thick](d10) to (d11); 
                  \draw [thick](d11) to (d12); 
                        \draw [thick](d12) to (d10); 
                     
 \node[dot] (e07)  [label={[xshift=-0cm, yshift=0cm]$ba$}] at ($ (d05)  + (180+45:2)$) {};
\node[dot] (e08)  [label={[xshift=.4cm, yshift=-.2cm]$bab$}] at ($ (e07)   + (180+75:2)$) {};
\node[dot] (e09)  [label={[xshift=0cm, yshift=0cm]$bab^2$}] at ($ (e07)   +(180+15:2)$) {};

\node[dot] (e10)  [label={[xshift=.3cm, yshift=0cm]$b^2a$}] at ($ (d06)  + (180-45:2)$) {};
\node[dot] (e11)  [label={[xshift=.5cm, yshift=-.1cm]$b^2ab$}] at ($ (e10)   + (180-75:2)$) {};
\node[dot] (e12)  [label={[xshift=0cm, yshift=-.75cm]$b^2ab^2$}] at ($ (e10)   +(180-15:2)$) {};

         \draw [thick](d05) to (e07); 
            \draw [thick](e07) to (e08); 
                  \draw [thick](e07) to (e09); 
                        \draw [thick](e08) to (e09); 
                              \draw [thick](d06) to (e10); 
            \draw [thick](e10) to (e11); 
                  \draw [thick](e11) to (e12); 
                        \draw [thick](e12) to (e10);

\node[dot] (f08)  [label={[xshift=-.5cm, yshift=-.2cm]$baba$}] at ($ (e08)   + (210+75:2)$) {};
\node[dot] (f09)  [label={[xshift=.2cm, yshift=-.7cm]$bab^2a$}] at ($ (e09)   +(150+15:2)$) {};

\node[dot] (f11)  [label={[xshift=-.6cm, yshift=-.2cm]$b^2aba$}] at ($ (e11)   + (150-75:2)$) {};
\node[dot] (f12)  [label={[xshift=0cm, yshift=0cm]$b^2ab^2a$}] at ($ (e12)   +(210-15:2)$) {};

                            \draw [thick](e08) to (f08); 
              \draw [thick](e09) to (f09); 
                         \draw [thick](e11) to (f11); 
                                    \draw [thick](e12) to (f12);

\node[dot] (p09)  [label={[xshift=-0cm, yshift=0cm]$abab^2a$}] at ($ (d09)   + (-15:2)$) {};
\node[dot] (p08)  [label={[xshift=.6cm, yshift=-.1cm]$ababa$}] at ($ (d08)  +(105:2)$) {};

\node[dot] (p11)  [label={[xshift=-.7cm, yshift=-.3cm] $ab^2aba$}] at ($ (d11)   +(-105:2)$) {};
\node[dot] (p12)  [label={[xshift=0cm, yshift=-.7cm]$ab^2ab^2a$}] at ($ (d12)   +(15:2)$) {};

                            \draw [thick](p09) to (d09); 
              \draw [thick](p08) to (d08); 
                         \draw [thick](p11) to (d11); 
                                    \draw [thick](p12) to (d12);

\node (p09x)  at ($ (p09)   + (-45:2)$) {};
\node (p09y)   at ($ (p09)   + (15:2)$) {};
   \draw[dotted] (p09x) to (p09); 
      \draw[dotted] (p09y) to (p09);

      \node (p08x)  at ($ (p08)   + (140:2)$) {};
\node (p08y)   at ($ (p08)   + (80:2)$) {};
   \draw[dotted,] (p08x) to (p08); 
      \draw[dotted] (p08y) to (p08);

      \node (p11x)  at ($ (p11)   + (-135:2)$) {};
\node (p11y)   at ($ (p11)   + (-75:2)$) {};
   \draw[dotted,] (p11x) to (p11); 
      \draw[dotted] (p11y) to (p11);

      \node (p12x)  at ($ (p12)   + (45:2)$) {};
\node (p12y)   at ($ (p12)   + (-15:2)$) {};
   \draw[dotted,] (p12x) to (p12); 
      \draw[dotted] (p12y) to (p12);

\node (f09x)  at ($ (f09)   + (180-45:2)$) {};
\node (f09y)   at ($ (f09)   + (180+15:2)$) {};
   \draw[dotted] (f09x) to (f09); 
      \draw[dotted] (f09y) to (f09);

      \node (f08x)  at ($ (f08)   + (180+140:2)$) {};
\node (f08y)   at ($ (f08)   + (180+80:2)$) {};
   \draw[dotted,] (f08x) to (f08); 
      \draw[dotted] (f08y) to (f08);

      \node (f11x)  at ($ (f11)   + (180-135:2)$) {};
\node (f11y)   at ($ (f11)   + (180-75:2)$) {};
   \draw[dotted,] (f11x) to (f11); 
      \draw[dotted] (f11y) to (f11);

      \node (f12x)  at ($ (f12)   + (180+45:2)$) {};
\node (f12y)   at ($ (f12)   + (180-15:2)$) {};
   \draw[dotted,] (f12x) to (f12); 
      \draw[dotted] (f12y) to (f12);

 \end{tikzpicture}

	\caption{Part of the undirected Cayley graph $\Gamma(G,\{a,b\})$ for $G=C_2\ast C_3$ with presentation $\langle a,b\mid a^2=1, b^3=1\rangle$. }
	\label{fig:example23}
\end{figure}

\begin{remark}\label{rmk:LargeIECs}
Note that  the undirected Cayley graph for the group $G=C_2\ast C_3$ shown in Figure~\ref{fig:example23} is geodetic, and  isometrically embedded circuits have length at most 3. If we consider $G_{2n+1}=C_2\ast C_{2n+1}$ with presentation $\langle a,b\mid a^2=1, b^{2n+1}=1\rangle$ for arbitrarily $n\in\N$, the undirected Cayley graph is geodetic and has isometrically embedded circuits of length at most $2n+1$. 
 This family of examples shows that geodetic Cayley graphs may contain isometrically embedded circuits of any (odd) length. By Corollary~\ref{cor:Plain} such groups are presented by  finite convergent length-reducing rewriting systems.

\end{remark}

 \section{ Embedded circuits in geodetic graphs}\label{sec:GraphTheorem}
 
In this section we prove   Theorem~\ref{ThmB:Shortcircuits}. We start with the following  lemma.

\begin{lemma}\label{lem:level3}
Let $\Gamma$ be a simple geodetic graph.
If $\rho$ is an  embedded circuit of diameter exceeding two and that has minimal length among all such embedded circuits in $\Gamma$, then  $\rho$ contains a geodesic sub-path of length three. \end{lemma}

\begin{proof}
Let $\rho$ be an   embedded circuit of diameter exceeding two and that has minimal length among all  embedded circuits of diameter exceeding two in $\Gamma$.  
Since $\rho$ has diameter at least three, there exist vertices $1$ and $x$ visited by $\rho$ such that $d(1, x) = 3$.  We choose a basepoint (the vertex $1$), an orientation of $\rho$, and label the vertices visited by $\rho$ in order
\[1, u_1, u_2, \dots, u_m = x = v_n, v_{n-1}, \dots, v_1, 1.\]
For each vertex $w \in \Gamma$, we say that $w$ is in \emph{level} $d(w, 1)$.  

Note that $m,n\geq 3$ since $\rho$ has diameter at least three.

 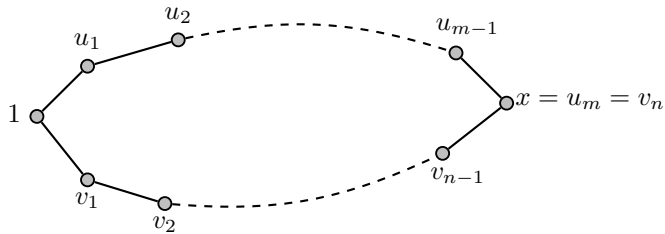
\begin{figure}[!ht]
	\centering

\begin{tikzpicture}[auto,node distance=.5cm,
    latent/.style={circle,draw,very thick,inner sep=0pt,minimum size=30mm,align=center},
    manifest/.style={rectangle,draw,very thick,inner sep=0pt,minimum width=45mm,minimum height=10mm},
    paths/.style={->, ultra thick, >=stealth'},
]

    \node[dot] [label={[xshift=-.3cm, yshift=-.3cm]$1$}](0) {};
    \node[dot] (1) [label={[xshift=0cm, yshift=-0cm]$u_1$}, above right= 1.5em and 1.5em of 0]  {};
        \node[dot] (2) [label={[xshift=0cm, yshift=0cm]$u_2$}, above right = .6em and 3em of 1]  {};
                \node[dot] (3) [label={[xshift=.15cm, yshift=0cm]$u_{m-1}$},  below  right = .1em and 10em of 2]  {};
                                \node[dot] (4) [label={[xshift=--1.1cm, yshift=-.3cm]$x=u_m=v_n$},below right = 1.5em and 1.5em of 3]  {};
    \node[dot] (11)  [label={[xshift=0cm, yshift=-.6cm]$v_1$}, below right = 2em and 1.5em of 0] {};
        \node[dot] (12)  [label={[xshift=0cm, yshift=-.6cm]$v_2$},below right = .5em and 2.5em of 11] {};
 \node[dot] (14)   [label={[xshift=.2cm, yshift=-.65cm]$v_{n-1}$}, below left= 1.5em and 2em of 4] {};

    \path[draw,thick]
    (0) edge node {} (1)
     (1) edge node {} (2)
               (3) edge node {} (4)
                         (4) edge node {} (14)
    (0) edge node {} (11)
      (11) edge node {} (12);
        
            \draw [thick,dashed] (2) to [bend left=15]  (3)
            (12) to [bend right=15]  (14);

\end{tikzpicture}

	\caption{The embedded circuit $\rho$ in Lemma~\ref{lem:level3}. }
	\label{fig:length3geod}
\end{figure}

\noindent 
\textit{Claim 1:} $u_2,v_2$ are in level 2. \\
 First we note that, since $\rho$ is an embedded circuit, the vertices $1, \dots u_{m-1}, v_1, \dots, v_{n-1}, x$ are distinct. Since $1$ and $u_1$ are distinct, $u_1$ is in level 1.  Suppose that $u_2$ is not in level 2.  Then it is either in level $0$ or $1$, but $u_2\neq 1$ so it must be in level $1$.
This implies that $u_2$ is adjacent to $1$, and omitting $u_1$ from $\rho$ yields a shorter embedded circuit of diameter exceeding two.  This contradicts the choice of 
$\rho$, and hence proves that $u_2$ is in level 2.

A symmetric argument shows that $v_2$ is in level 2.

\medskip

Since $\Gamma$ is geodetic, $u_1$ is the unique level-1 vertex adjacent to $u_2$.  It follows that $u_3$ is in level 2 or level 3.  Similarly, $v_3$ is in level 2 or level 3.  The result is proved if we can show that $u_3$ and $v_3$ cannot both be in level 2.

\smallskip
\noindent 
\textit{Claim 2:} At least one of $u_3,v_3$ is in level 3.\\
Suppose that $u_3$ and $v_3$ are both in level 2.  Let $u'_1$ be the unique vertex in level 1 that is adjacent to $u_3$; let $v'_1$ be the unique vertex in level 1 that is adjacent to $v_3$.  
If  $\rho$ does not visit $u'_1$, then replacing the subpath $1,u_1,u_2,u_3$ by the path $1,u_1',u_3$ yields a shorter embedded circuit of diameter at least three, contradicting our choice of $\rho$. Therefore 
$\rho$ visits $u'_1$.
If $u'_1 \neq v_1$, then either $1, v_1, \dots u_1',1$ or $1, u_1, \dots, u_1', 1$ is an embedded circuit of diameter at least 3, contradicting our choice of $\rho$. 
Thus 
 $u'_1 = v_1$.  By a symmetric argument, we also have $v_1'=u_1$, and we are now in the situation shown in Figure~\ref{fig:length3geod}.

 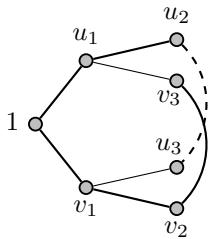
\begin{figure}[!ht]
	\centering

\begin{tikzpicture}[auto,node distance=.5cm,
    latent/.style={circle,draw,very thick,inner sep=0pt,minimum size=30mm,align=center},
    manifest/.style={rectangle,draw,very thick,inner sep=0pt,minimum width=45mm,minimum height=10mm},
    paths/.style={->, ultra thick, >=stealth'},
]

  \node[dot] [label={[xshift=-.3cm, yshift=-.3cm]$1$}](0) {};

    \node[dot] (1) [label={[xshift=0cm, yshift=0cm]$u_1$}, above right= 2em and 1.5em of 0]  {};
        \node[dot] (11) [label={[xshift=0cm, yshift=-.6cm]$v_1$}, below right = 2em and 1.5em of 0] {};
        \node[dot] (2) [label={[xshift=0cm, yshift=0cm]$u_2$}, above right = .4em and 3em of 1]  {};
              \node[dot] (12) [label={[xshift=0cm, yshift=-.6cm]$v_2$}, below right = .4em  and  3em  of 11] {};    
                \node[dot] (13) [label={[xshift=-.1cm, yshift=-.55cm]$v_3$},below = 1em of 2]  {};
            \node[dot] (3)  [label={[xshift=-.1cm, yshift=-.05cm]$u_3$}, above= 1em of 12] {};

    \path[draw]
  (1) edge node {} (13)
             (11) edge node {} (3);

    \path[draw,thick]
    (0) edge node {} (1)
      (0) edge node {} (11)
     (1) edge node {} (2)
      (11) edge node {} (12);

             \draw [thick,dashed] (2) to [bend left=45]  (3);
             \draw [thick]    (12) to [bend right=45]  (13)  ;
             
\end{tikzpicture}

	\caption{Case $u'_1 = v_1$ and $v'_1 = u_1$  in Lemma~\ref{lem:level3}.}
	\label{fig:length3geod}
\end{figure}

Let $\rho'$ be obtained from $\rho$ by replacing $1, u_1, u_2, u_3$ by $1, v_1, u_3$, and replacing $v_3, v_2, v_1, 1$ by $v_3, u_1, 1$.  Since $\rho'$ visits only vertices visited by $\rho$, and 1 is the only vertex visited twice, we know that $\rho'$ is an embedded circuit which is shorter than $\rho$.  Since the only vertices from $\rho$ omitted were in levels 1 and 2, we know that $\rho'$ still visits a vertex in level 3, and hence it still has diameter at least 3,  contradicting our choice of $\rho$. 
\end{proof}

We will make use of the following fact due to Stemple. 

\begin{lemma}[{{\cite[Theorem 3.3]{Stemple}}}]\label{lem:Stemple4}
If a geodetic graph contains an embedded circuit \[w_0, w_1, w_2, w_3, w_0\] of length four, then the induced subgraph on these vertices is a complete graph.
\end{lemma}

Next we have the following technical result.

\begin{lemma}\label{lem:FormOfRho}
Let $\Gamma$ be a geodetic graph in which any \iec\ has at most five edges.
Suppose that $\rho$ is an embedded circuit in $\Gamma$ of diameter at least three, and $\rho$ has minimal length among all such embedded circuits.  Without loss of generality (using Lemma~\ref{lem:level3}), we may label the vertices of $\rho$ such that one traversal of $\rho$ reads
\[1=u_0, u_1, \dots, u_m=v_3, v_2, v_1, 1\] and $1, v_1, v_2, v_3$ is a geodesic subpath.   Then $m=5$, $d(1, u_1) = 1$, $d(1, u_2) = d(1, u_3) = 2$, $d(1, u_4) = 3$ and $d(u_3, v_1) = 1$. 
\end{lemma}

 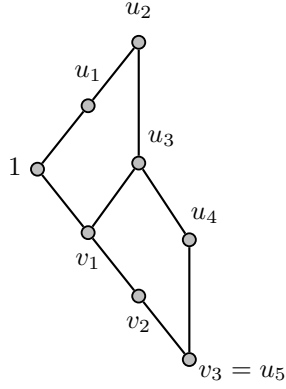
\begin{figure}[!ht]
	\centering

	\begin{tikzpicture}[auto,node distance=.5cm,
    latent/.style={circle,draw,very thick,inner sep=0pt,minimum size=30mm,align=center},
    manifest/.style={rectangle,draw,very thick,inner sep=0pt,minimum width=45mm,minimum height=10mm},
    paths/.style={->, ultra thick, >=stealth'},
]

   \node[dot] [label={[xshift=-.3cm, yshift=-.3cm]$1$}](0) {};
 
    \node[dot] (1)  [label={[xshift=0cm, yshift=.1cm]$u_1$}, above right= 2em and 1.5em of 0]  {};
        \node[dot] (2)  [label={[xshift=0cm, yshift=.1cm]$u_2$}, above right = 2em and 1.5em of 1]  {};
                \node[dot] (3)  [label={[xshift=.3cm, yshift=0cm]$u_3$}, below = 4em of 2]  {};
                        \node[dot] (4)  [label={[xshift=.2cm, yshift=0cm]$u_4$}, below right = 2.5em and 1.5em of 3]  {};
    \node[dot] (11)  [label={[xshift=0cm, yshift=-.7cm]$v_1$}, below right = 2em and 1.5em of 0] {};
        \node[dot] (12)  [label={[xshift=0cm, yshift=-.7cm]$v_2$}, below right = 2em and 1.5em of 11] {};
            \node[dot] (13)  [label={[xshift=--.7cm, yshift=-.5cm]$v_3=u_5$}, below right = 2em and 1.5em of 12] {};

    \path[draw,thick]
    (0) edge node {} (1)
     (1) edge node {} (2)
          (2) edge node {} (3)
               (3) edge node {} (4)
                         (4) edge node {} (13)
    (0) edge node {} (11)
      (11) edge node {} (12)
        (12) edge node {} (13)
             (11) edge node {} (3);
    
\end{tikzpicture}

	\caption{Conclusion of Lemma~\ref{lem:FormOfRho}.}
	\label{fig:conclusionFormOfRho}
\end{figure}

\begin{proof}
As before,  we say that a vertex $w$ is in \emph{level} $d(w, 1)$.  
Following 
 the proof of Lemma~\ref{lem:level3}, we have  that $u_1,v_1$ are in level 1, $u_2,v_2$ are in level 2, and $u_3,v_3$ are in level 2 or 3 but not both in level 2. We assumed without loss of generality in the hypothesis of this lemma that  $v_3$ that is in level 3.

\smallskip
\noindent 
\textit{Claim 1:}  $u_3$ is in level 2 and $d(v_1, u_3) =1$.  \\
Since $1,v_1,v_2,v_3$ is a geodesic and $\Gamma$ is geodetic, we  have $m\geq 4$, and  the path
 $1, u_1, \dots, u_m$ is not a geodesic.
 So, there exists a unique $i\leq m$ such that
 $1, u_1, u_2, \dots, u_{i-1}$ is geodesic and  $1, u_1, u_2, \dots, u_i$ is not a geodesic.
 It follows that $u_{i-1}$ and $u_i$ are both in level $i-1$.  Since $u_1$ is in level 1 and $u_2$ is in level 2, we know that $i \geq 3$.

 By 
 Lemma~\ref{lem:IECbroomlike},
 since $1,u_1,\dots, u_{i-1}$ is geodesic and $1,u_1, \dots, 
 u_{i-1}, u_i$ is not geodesic then by the $2$-broomlike property there is either  $u_{i-2}$ to $u_i$ are adjacent, or there is a geodesic from $u_{i-3}$ to $u_i$ of length 2. 
If $u_{i-2}$ and $u_i$ are adjacent, we could omit the vertex $u_{i-1}$ from the path $\rho$ and still have an embedded circuit that visits both 1 and $u_m=v_3$ --- a contradiction to our choice of $\rho$. Thus there is a geodesic from $u_{i-3}$ to $u_i$ of length 2.   It follows that there is a vertex $x\neq u_j$ for $0\leq j\leq i$ such that 
  $1,\dots, u_{i-3},x,u_i$ is a geodesic. See  Figure~\ref{fig:2-broomlike}.

  \begin{figure}[!ht]
	\centering

\begin{tikzpicture}[auto,node distance=.5cm,
    latent/.style={circle,draw,very thick,inner sep=0pt,minimum size=30mm,align=center},
    manifest/.style={rectangle,draw,very thick,inner sep=0pt,minimum width=45mm,minimum height=10mm},
    paths/.style={->, ultra thick, >=stealth'},
]

   \node[dot] [label={[xshift=-.3cm, yshift=-.3cm]$1$}](0) {};
    \node[dot] (1) [label={[xshift=0cm, yshift=0cm]$u_1$}, above right= 2em and 2em of 0]  {};
        \node[dot] (2) [label={[xshift=0cm, yshift=0cm]$u_{i-3}$},above right = 3em and 8em of 1]  {};
                \node[dot] (3) [label={[xshift=0cm, yshift=-.65cm]$x$}, below right = 2em and 3.5em of 2]  {};
                        \node[dot] (4) [label={[xshift=0cm, yshift=-.7cm]$u_{i}$},right = 2.5em of 3]  {};
                \node[dot] (5) [label={[xshift=0cm, yshift=0cm]$u_{i-2}$},right = 3em of 2]  {};
                        \node[dot] (6) [label={[xshift=0cm, yshift=0cm]$u_{i-1}$}, right = 2.5em of 5]  {};

    \path[draw,thick]
    (0) edge node {} (1)
          (2) edge node {} (3)
               (3) edge node {} (4)
                         (2) edge node {} (5)
    (5) edge node {} (6)
     (6) edge node {} (4);

             \draw [thick,dashed] (1) to [bend left=20]  (2);
      
\end{tikzpicture}
	\caption{Using the $2$-broomlike property in the proof of Claim 1 of  Lemma~\ref{lem:FormOfRho}.}
	\label{fig:2-broomlike}
\end{figure}
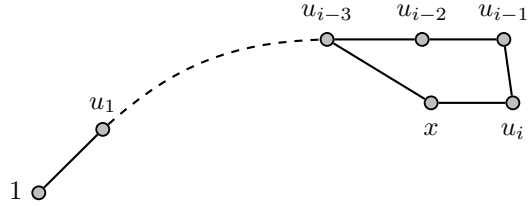

 Observe that replacing in $\rho$ the subpath $u_{i-3}, u_{i-2}, u_{i-1}, u_i$ with $u_{i-3}, x, u_{i}$ yields a closed path $\rho'$ that visits both 1 and $u_m = v_3$.  The minimality of the length of $\rho$ implies that $\rho'$ is not an embedded circuit; that is, $x$ must be  equal to one of the vertices of $\rho'$.  If $x = u_j$ for some $i+1\leq j\leq m$, then we can remove a cycle from $\rho'$ and construct a shorter embedded circuit that visits both 1 and $v_3$. It follows that either $x = v_1$ or $x = v_2$.  
 Suppose $x=v_2$. Then $x$ is in level 2 and so 
 $1, u_1, x$ is a geodesic, as is $1,v_1,x$, and since $u_1 \neq v_1$ we contradict that $\Gamma$ is geodetic.
 Hence we have that $x = v_1$.  This means that $i = 3$ and $u_3$ is in level 2, and $d(v_1, u_3) =1$, as required.

\smallskip
\noindent 
\textit{Claim 2:} $m\geq 5$.\\
We know that $m\geq 4$. If $m=4$ then $v_1,u_3,u_4$ and $v_1,v_2,v_3$ are two different geodesics between the same endpoints, contradicting geodecity. Thus $m\geq 5$.

\smallskip
\noindent 
\textit{Claim 3:} $u_4$ is in level 3. \\
 Since $\Gamma$ is geodetic, $v_1$ is the only vertex in level 1 that is adjacent to $u_3$.  Since $m\geq 4$ and  $u_4 \neq v_1$ (because $\rho$ is an embedded circuit), we have that $u_4$ is in level 2 or level 3.  Suppose that $u_4$ is in level 2, and let $p$ denote the unique vertex in level 1 that is adjacent to $u_4$.

  \begin{figure}[!ht]
	\centering

	\begin{tikzpicture}[auto,node distance=.5cm,
    latent/.style={circle,draw,very thick,inner sep=0pt,minimum size=30mm,align=center},
    manifest/.style={rectangle,draw,very thick,inner sep=0pt,minimum width=45mm,minimum height=10mm},
    paths/.style={->, ultra thick, >=stealth'},
]

   \node[dot] [label={[xshift=-.3cm, yshift=-.3cm]$1$}](0) {};
 
    \node[dot] (1)  [label={[xshift=0cm, yshift=.1cm]$u_1$}, above right= 2em and 1.5em of 0]  {};
        \node[dot] (2)  [label={[xshift=0cm, yshift=.1cm]$u_2$}, above right = 2em and 1.5em of 1]  {};
                \node[dot] (3)  [label={[xshift=.3cm, yshift=0cm]$u_3$}, below = 4em of 2]  {};
                        \node[dot] (4)  [label={[xshift=.2cm, yshift=0cm]$u_4$}, below right = 2.5em and 1.5em of 3]  {};
    \node[dot] (11)  [label={[xshift=0cm, yshift=-.65cm]$v_1$}, below right = 2em and 1.5em of 0] {};
        \node[dot] (12)  [label={[xshift=0cm, yshift=-.65cm]$v_2$}, below right = 2em and 1.5em of 11] {};
            \node[dot] (13)  [label={[xshift=0cm, yshift=-.65cm]$v_3$}, below right = 2em and 1.5em of 12] {};

    \path[draw,thick]
    (0) edge node {} (1)
     (1) edge node {} (2)
          (2) edge node {} (3)
               (3) edge node {} (4)                        
    (0) edge node {} (11)
      (11) edge node {} (12)
        (12) edge node {} (13)
             (11) edge node {} (3);

             \draw [thick,dashed] (4) to [bend left=90]  (13); 
\end{tikzpicture}

	\caption{Claim 3 in the proof of Lemma~\ref{lem:FormOfRho}: assume $u_4$ is in level 2.}
	\label{fig:claim2}
\end{figure}
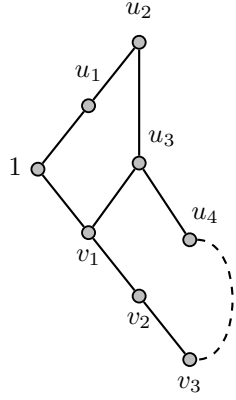

  Now either $p$ is a vertex of $\rho$, or not. If it does not lie on $\rho$ then we can replace the subpath $1,u_1,u_2,u_3,u_4$ by $1,p,u_4$ and obtain a shorter embedded circuit which visits $1$ and $v_3$, contradicting the minimality of $\rho$.
 
  Therefore $p$ is a vertex of $\rho$.

\smallskip
\noindent 
\textit{Case 1:}  $p = v_1$.  \\
The path
$1, u_1, u_2, u_3, u_4, v_1, 1$
is an embedded circuit of length 6. Call this path $\tau$.
Since $u_1,u_4$ are distinct we have $1\leq  d(u_1,u_4)\leq 3$.
The paths $u_1,u_2,u_3,u_4$ and $u_1, 1, v_1, u_4$ are both length 3, so 
$d(u_1,u_4)\neq 3$ or the graph is not geodetic.
If $d(u_1,u_4)=1$ then we can replace in $\rho$ the path $u_1,u_2,u_3,u_4$ by $u_1,u_4$ and find a shorter embedded circuit of diameter exceeding 2. Thus $d(u_1,u_4)=2$.

It follows that there must be a vertex $t$ that is not visited by $\tau$ and is adjacent to both $u_1$ and $u_4$.  Since $t$ does not lie on $\tau$, $t\neq v_1$, and since $u_4$ is in level 2, $t$ is in level 2 (if it were in level 1 we would have two geodesics to $u_4$ contradicting geodecity).  Since $t$ is adjacent to $u_1$ and $v_1\neq u_1$, if  $t =v_2$ we would have two geodesics to $v_2$, thus $t\neq v_2$.  It follows that by replacing in $\rho$ the subpath $1, u_1, u_2, u_3, u_4$ by the path $1, u_1, t, u_4$, and removing a subpath that is a cycle if necessary, we may construct a shorter embedded circuit that visits both 1 and $v_3$.  This contradiction proves that this case is impossible.

\smallskip
\noindent 
\textit{Case 2:} $p = u_1$. \\
Omitting $u_2$ and $u_3$ from $\rho$ would yield a shorter embedded circuit that still visits 1 and $v_3$. This contradiction proves that this case is impossible.

\smallskip
\noindent 
\textit{Case 3:}
 $u_1 \neq p \neq v_1$.  \\
 In this case  $p=u_j$ for $5\leq j<m$ since $v_1,v_2,v_3$ are all spoken for (only $v_1$ is in level 1).
 
 Then the path $1,p=u_j,u_{j+1},\dots, u_m=v_3,v_2,v_1,1$ is an embedded circuit passing $1$ and $v_3$ so has diameter 3 and is shorter than $\rho$, a contradiction.

Since all cases are impossible, we conclude that $u_4$ is not at level 2.  Hence $u_4$ is at level 3.

\smallskip
\noindent 
\textit{Claim 4:}  $u_5$ is at level 3. \\
 Since $u_3$ is the unique vertex in level 2 adjacent to $u_4$, and $u_5 \neq u_3$, we have that $u_5$ is not in level 2.  
 
 Suppose that $u_5$ is in level 4, and so  $\alpha=1,v_1,u_3,u_4,u_5$  is a geodesic.
 Since $\alpha$ is geodesic and  $1, v_1, u_3, u_4,u_5,\dots, u_m$ is not a geodesic, there exists a unique integer $i\geq 5$ so that $1, v_1, u_3, u_4,u_5,\dots, u_i$ is geodesic and $1, v_1, u_3, u_4,u_5, \dots, u_i,u_{i+1}$  is not geodesic, and  $u_{i}$ and $u_{i+1}$ are both in level $i-1$.

If $u_{i-1},u_i,u_{i+1}$ is not geodesic, omitting $u_i$ from  $\rho$ gives a shorter 
isometrically embedded circuit visiting $1$ and $v_3$, contradiction.
So $u_{i-1},u_i,u_{i+1}$ is geodesic.
 By 
 Lemma~\ref{lem:IECbroomlike}, 
 we must have 
 $u_{i-2},u_{i-1},u_i,u_{i+1}$ is not geodesic and there is a geodesic path $u_{i-2},z,u_{i+1}$
 where 
  $u_{i-2} \neq z \neq u_{i-1}$. 
  
 Note that  by construction $z$ is in level $i-2\geq 3$, so $z$ cannot equal $v_1,v_2$. 
 
 If  $z= v_3=u_m$ then $1, u_1,\dots, u_{i-2}, z, v_2,v_1, 1$ is a shorter isometrically embedded circuit that visits $1$ and $v_3$, a  contradiction. Also note that $z\neq u_6$ since $i\geq 5$ and $u_{i+1}\neq z$.

It follows that if $z=u_j$ then  $6<j\leq m-1$, and replacing the subpath  $u_{i-2},u_{i-1},u_i,u_{i+1}$ by $u_{i-2},z,u_{i+1}$ and possibly removing a cycle, we get a shorter embedded circuit than $\rho$ that visits $1$ and $v_3=u_m$.

 This shows that $z$ is not a vertex of $\rho$. Then replacing $u_{i-2},u_{i-1},u_i,u_{i+1}$ by $u_{i-2},z,u_{i+1}$ again gives a 
 shorter embedded  circuit than $\rho$ that visits $1$ and $v_3$.

  This contradiction proves that $u_5$ is in level 3.

\smallskip
\noindent 
\textit{Claim 5:}  $m=5$.\\
We note that $u_5$ is not adjacent to $u_2$ or $u_3$, otherwise we could omit $u_3,u_4$ or $u_4$ respectively from $\rho$ and have a shorter embedded circuit that visits both 1 and $v_3$.  Since $u_5$ is in level 3, we have  $v_1, u_3, u_4$ is a geodesic and $v_1, u_3, u_4, u_5$ is not geodesic, and $u_3,u_5$ is not an edge so 
 Lemma~\ref{lem:IECbroomlike} implies  there exists a vertex $q$ adjacent to both $v_1$ and $u_5$ such that $u_2 \neq q \neq u_3$.  Therefore $1, u_1, u_2, u_3, u_4, u_5, q, v_1, 1$ is an embedded circuit visiting $1$ and a vertex at level 3, so 
by  the minimality of  $\rho$ we must have  $q=v_2$ and $ u_5=v_3$. 
 \end{proof}

We can now prove
Theorem~\ref{ThmB:Shortcircuits}.

\begin{proof}[Proof of Theorem~\ref{ThmB:Shortcircuits}]
Suppose that there exists in $\Gamma$ an embedded circuit of diameter exceeding two.  By Lemma \ref{lem:FormOfRho}, there exists an embedded circuit $\rho$  labeled \[1,u_1,u_2,u_3,u_4,v_3,v_2,v_1,1\] with
$u_1$ at level 1, $u_2,u_3$ at level 2, $u_4$ at level 3 and $d(u_3,v_1)=1$, and $1,v_1,v_2,v_3$ is geodesic, as illustrated in Figure~\ref{fig:conclusionFormOfRho}.
Let $\rho'$ be the embedded circuit that begins at $v_3$ and visits the same vertices as $\rho$, but in reverse order.  That is, $\rho'$ visits vertices in the following order
\[v_3, u_4, u_3, u_2, u_1, 1, v_1, v_2, v_3.\]  Now $\rho'$ is also a minimal length embedded circuit with diameter exceeding two, so  Lemma \ref{lem:FormOfRho} applies to $\rho'$ as well (with $u_2$ playing the role of $u_3$ and $v_2$ the role of $v_1$), which gives that $d(u_2, v_2) = 1$.

  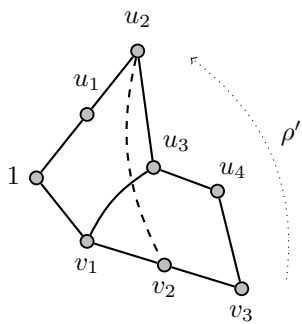
\begin{figure}[!ht]
	\centering
	\begin{tikzpicture}[auto,node distance=.5cm,
    latent/.style={circle,draw,very thick,inner sep=0pt,minimum size=30mm,align=center},
    manifest/.style={rectangle,draw,very thick,inner sep=0pt,minimum width=45mm,minimum height=10mm},
    paths/.style={->, ultra thick, >=stealth'},
]

   \node[dot] [label={[xshift=-.3cm, yshift=-.3cm]$1$}](0) {};
      \node[dot] (1)  [label={[xshift=0cm, yshift=.1cm]$u_1$}, above right= 2em and 1.5em of 0]  {};
        \node[dot] (2)  [label={[xshift=0cm, yshift=.1cm]$u_2$}, above right = 2em and 1.5em of 1]  {};
            \node[dot] (11)  [label={[xshift=0cm, yshift=-.65cm]$v_1$}, below right = 2em and 1.5em of 0] {};

                \node[dot] (3) [label={[xshift=.3cm, yshift=0cm]$u_3$},below right = 4em and .2em of 2]  {};
                          \node[dot] (4)  [label={[xshift=.2cm, yshift=-.05cm]$u_4$}, below right = .5em and 2em of 3]  {};

        \node[dot] (12)  [label={[xshift=0cm, yshift=-.65cm]$v_2$}, below right = .5em and 2.5em of 11] {};
            \node[dot] (13)   [label={[xshift=0cm, yshift=-.65cm]$v_3$}, below right = .5em and 2.5em of 12] {};

    \path[draw,thick]
    (0) edge node {} (1)
     (1) edge node {} (2)
          (2) edge node {} (3)
               (3) edge node {} (4)
                         (4) edge node {} (13)
    (0) edge node {} (11)
      (11) edge node {} (12)
        (12) edge node {} (13);

             \draw [thick,dashed] (2) to [bend right=20]  (12);
                  \draw [thick] (11) to [bend left=15]  (3);

\node (e1)   [ right = 1em  of 13] {};
\node (e2)   [ right = 1em  of 3] {};                    
\node (e3)   [ right = 1em  of 2] {};    

        \draw [->,dotted] (e1) to [bend right=30]  node [above right]  (e2) {$\rho'$} (e3);

\end{tikzpicture}

	\caption{The path $\rho'$ which starts at $v_3$ and runs in the reverse direction to $\rho$ in the proof of Theorem~\ref{ThmB:Shortcircuits}.}
	\label{fig:TheoremB}
\end{figure}

It follows that $u_2, v_2, v_1, u_3, u_2$ is an embedded circuit of length 4.  By 
Lemma~\ref{lem:Stemple4}, we must have that $u_2$ and $v_1$ are adjacent.  This contradicts the fact that $u_1$ is the unique level-1 vertex adjacent to $u_2$.
This contradiction proves that there are no embedded circuits in $\Gamma$ with diameter exceeding two. 
\end{proof}

\section{Plain groups, blocks and embedded circuits}\label{sec:plain}

Bass-Serre theory \cite{KPS1973,Trees} tells us that a group $G$ is plain if and only if $G$ acts
on a locally-finite tree, with a compact quotient, finite vertex stabilisers, trivial edges stabilisers and no edge inversions. See for example \cite[Theorem 13]{Trees}.

 Another useful characterisation of plain groups then follows from 
the {block-cut tree} associated to the graph, described in Section~\ref{sec:GraphTheory}.

For a finite set of vertices $S$ in a graph  $\Gamma$, the diameter of $S$ is the maximum distance in $\Gamma$ between any pair of vertices in $S$.
 Haring-Smith \cite{HS83} proved the following result in 1983.  We provide a short proof that uses Bass-Serre theory and the block-cut tree.

\begin{theorem}[A characterisation of plain groups]\label{thm:PlainAndTwoBlocks}
For a group $G$ and a positive integer $s$, the following are equivalent:
\begin{enumerate}

        \item $G$ admits a finite generating set $\Sigma$ such that, in the associated undirected Cayley graph $\Gamma(G,\Sigma)$,  the diameter of any embedded circuit is at most $s$. 
    
       \item $G$ admits a finite generating set $\Sigma$ such that, in the associated undirected Cayley graph $\Gamma(G,\Sigma)$,  the diameter of any block is at most $s$. 
     
    \item $G$ is a plain group.
\end{enumerate}
\end{theorem}

\begin{proof}
$1.\Leftrightarrow 2.$:  Follows immediately from  Lemma~\ref{lem:Whitman}.

$3.\Rightarrow 2.$: 
Suppose that $G$ is a plain group.  Then $G$ is a free product of $m$  finite groups $G_1, \dots, G_m$ and $n$ copies of the infinite cyclic group $C_1, \dots, C_n$. Let $\Sigma$ be a set comprising each nontrivial element of each finite factor $G_i$, and one generator $a_i$ and its inverse $A_i$ for each infinite cyclic factor $C_i$. In the Cayley graph $\Gamma = \Gamma(G, \Sigma)$, the only blocks containing the identity element $e_G$ are the subgraphs induced by $\Gamma(G_i, G_i\setminus\{e_{G_i}\})$  for $1\leq i\leq m$ (and these are complete graphs), and subgraphs induced by $(e_G, a_i)$ for $1 \leq i \leq n$. Thus all blocks containing $e_G$ have diameter 1.  Since $\Gamma$ is vertex-transitive, all blocks in $\Gamma$ have diameter one (and hence all blocks in $\Gamma$ have diameter at most $s$).

$2.\Rightarrow 3.$: 
Suppose that $G$ admits a finite generating set $\Sigma$ such that in the associated Cayley graph $\Gamma=\Gamma(G,\Sigma)$ all blocks have  diameter at most  $s$.  Let $T$ denote the block-cut tree of $\Gamma$, as described in Section~\ref{sec:GraphTheory}.  The natural left-action of $G$ on $\Gamma$ induces a left-action of $G$ on $T$.  Since the action of $G$ on $\Gamma$ is vertex transitive, the action of $G$ on $T$ is transitive on the set of type I vertices and there are finitely many orbits of type II vertices.  It follows that the action of $G$ on $T$ is  cocompact.  In the action of $G$ on $\Gamma$, vertices have trivial stabilisers.  It follows that in the action of $G$ on $T$, type I vertices have trivial stabilisers, type II vertices have finite stabilisers (because blocks in $\Gamma$ comprise finitely many vertices), edges are not inverted (each edge includes a type I and type II vertex which cannot be interchanged) and edge stabilisers are trivial.   Since $G$ acts  on $T$, a locally-finite tree, with finite vertex stabilisers, trivial edges stabilisers and no edge inversions, by \cite[Theorem 13]{Trees}  
 $G$ is a plain group.
\end{proof}

If a rewriting system $(\Sigma, T)$ presents a group $G$, then properties of the rewriting system determine properties of the Cayley graph $\Gamma = \Gamma(G, \Sigma)$.

\begin{lemma}\label{lem:RewritingAndGamma}
Let $(\Sigma, T)$ be a  finite convergent length-reducing rewriting system such that $\Sigma = \Sigma^{-1}$ and $(\Sigma, T)$ presents a group $G$.  Let $\Gamma$ denote the undirected Cayley graph of $\Gamma$  with respect to $\Sigma$. Then 
\begin{enumerate}
\item $\Gamma$ is geodetic;
\item  If $u_0, u_1,\dots, u_{m-1},u_m=u_0$ 
is an \iec\ in $\Gamma$ of length $m>2$, then
 $m = 2n+1$ for some positive integer $n$ and $(x_1 \dots x_{n+1}, x_m^{-1}\dots x_{n+2}^{-1} ) \in T$ where $x_i=_Gu_{i-1}^{-1}u_i\in\Sigma$ for $1\leq i\leq m$.

\end{enumerate}
\end{lemma}
\begin{proof}
If $u_0,\dots, u_n$ and $v_0,\dots, v_n$ are two geodesics in $\Gamma(G,\Sigma)$ with $u_0=v_0,u_n=v_n$, then the words 
\[u=(u_0^{-1}u_1)\dots(u_{n-1}^{-1}u_n)\in \Sigma^*\ \text{  and  }\  v=(v_0^{-1}v_1)\dots(v_{n-1}^{-1}v_n)\in \Sigma^*\]
are irreducible words representing the same group element.  By Lemma \ref{lem:ConvergentRewriting}, $u=v$,  which establishes the first claim.

If 
$u_0, u_1,\dots, u_{m-1},u_m=u_0$ 
is an \iec\ in $\Gamma$ of length $m>2$,
set $x_i=_Gu_{i-1}^{-1}u_i\in\Sigma$ for $1\leq i\leq  m$.
If  $m=2n$ then 
$u_0,\dots, u_n$ and $u_0, u_{m-1},\dots, u_n$
are two geodesics for the same element, and since the circuit is embedded and $m>2$, so $n>1$, we have $u_1\neq u_{m-1}$,
so $a,b$ are distinct words, which contradicts the first claim.
Thus $m=2n+1$.

Now let $a=x_1\dots x_{n+1}$ and  $b=x_m^{-1}\dots x_{n+2}^{-1} $. Then   $a=_Gb$.
The word $b$ is geodesic since it is a subpath of length $n$ of an IEC of length $2n+1$.
The word $a$ is not geodesic  
so some rewrite rule must apply. We have  $a=u\ell v$ with  $(\ell,r)\in T$,  $|r|<|\ell|$ and $a=_G urv$.
If $|u|+|v| > 0$, then $\ell$ is geodesic since it is a subpath of length at most $n$ of  an IEC of length $2n+1$.
Then $\ell\neq_G r$ for any $r \in\Sigma^*$ with $|r| < |\ell|$. Hence
$u = v = \lambda$ and $a= \ell$. But then $r = b$ because $b$, being geodesic and shorter than $a$ by one letter, is the unique word $r$ with $|r|<|a|$ and $r =_G a$.

\end{proof}

We are now ready to prove the main theorem.

\begin{proof}[Proof of Theorem~\ref{thmA:main}]
Corollary \ref{cor:Plain} gives one direction.   

Suppose that $G$ admits presentation by a finite convergent length-reducing rewriting system $(\Sigma, T)$ such that 
$\Sigma=\Sigma^{-1}$ and the left-hand side of every rule has length at most three.  
Let $\Gamma$ be the undirected Cayley graph of $G$ with respect to $\Sigma$. By Lemma \ref{lem:RewritingAndGamma}, $\Gamma$ is geodetic and \iec s have length at most five. Since $\Gamma$ satisfies the hypotheses of Theorem~\ref{ThmB:Shortcircuits}, all embedded circuits in $\Gamma$ have diameter at most two.  By Theorem \ref{thm:PlainAndTwoBlocks}, $G$ is plain.
\end{proof}

\section*{Acknowledgments}
Research supported by  Australian Research Council grant  DP210100271.
The authors thank the anonymous reviewer for helpful feedback and corrections.

\bibliographystyle{cas-model2-names}
\bibliography{LengthReducingRewritingBib}

\end{document}